\documentclass{amsart}

\usepackage{palatino}
\usepackage{amsthm,amsmath,amssymb,latexsym,oldlfont}
\usepackage{pifont}
\usepackage[all]{xy}
\usepackage{mathtools}

\def\eu{\mathfrak}
\def\ma{\mathbb}

\def\p{{\mathcal P}_{\infty}}
\def\pp{{\eu p}_{\infty}}
\def\fin{\hfill\qed\bigskip}

\newcommand{\Irr}{\operatorname{Irr}}
\newcommand{\Gal}{\operatorname{Gal}}
\newcommand{\ex}{\operatorname{ex}}

\newcommand{\mcm}{\operatorname{lcm}}
\newcommand{\DrinT}{\operatorname{DrinT}}
\newcommand{\tor}{\operatorname{tor}}
\newcommand{\Hom}{\operatorname{Hom}}
\newcommand{\car}{\operatorname{char}}

\newcommand{\an}{\operatorname{an}}
\newcommand{\cog}{\operatorname{Drincog}}
\newcommand{\sgn}{\operatorname{sgn}}
\newcommand{\Pic}{\operatorname{Pic}}

\DeclareMathOperator{\grado}{deg}
\DeclareMathOperator{\cogalois}{Drincog}
\DeclareMathOperator{\cogaloistor}{Drincog}
\DeclareMathOperator{\grugal}{Gal}

\newcounter{bean}

\def\las{\begin{list}
	{{\rm {(\arabic{bean})}}}{\usecounter{bean}
	\setlength{\rightmargin}{\leftmargin}}}

\newtheorem{teorema}{Theorem}[section]
\newtheorem{proposition}[teorema]{Proposition}
\newtheorem{lema}[teorema]{Lemma}
\newtheorem{corollary}[teorema]{Corollary}

\theoremstyle{definition}
\newtheorem{ejemplo}[teorema]{Example}
\newtheorem{definicion}[teorema]{Definition}
\newtheorem{observacion}[teorema]{Remark}

\title[Cogalois theory and Drinfeld modules]
{Cogalois theory and Drinfeld modules}

\author[M. S\'anchez]{Marco Antonio S\'anchez--Mirafuentes}
\address{Departamento de Control Autom\'atico\\
Centro de Investigaci\'on y de Estudios Avanzados del I.P.N.}
\email{kmasm1969@yahoo.com.mx}

\author[J.C. Salas]{Julio Cesar Salas--Torres}
\address{Departamento de Control Autom\'atico\\
Centro de Investigaci\'on y de Estudios Avanzados del I.P.N.}
\email{jcstorres88@hotmail.com, torres1jcesar0@gmail.com}

\author[G. Villa]{Gabriel Villa--Salvador}
\address{Departamento de Control Autom\'atico\\
Centro de Investigaci\'on y de Estudios Avanzados del I.P.N.}
\email{gvillasalvador@gmail.com, gvilla@ctrl.cinvestav.mx}

\subjclass[2010]{Primary 1R60; Secoundary 11R18, 11R32, 11R58}

\keywords{Drinfeld modules; cogalois theory; torsion in Drinfeld modules;
cyclotomic function fields; cogalois groups}

\date{May 9th., 2016}

\begin{document}

\begin{abstract}

In this paper we generalize the results of
\cite{sanchez} to rank one Drinfeld modules with
class number one. We show that, in the present form,
there does not exist a cogalois theory for Drinfeld modules
of rank or class number larger than one. We also consider
the torsion of the Carlitz module for the extension
${\mathbb{F}}_{q}(T)(\Lambda_{P^{n}})
/{\mathbb{F}}_{q}(T)(\Lambda_{P})$.

\end{abstract}

\maketitle

\section{Introduction}\label{S1}

The main goal of this paper is to obtain the analogue of the
classic cogalois group. The {\em cogalois group} of an 
arbitrary field extension $L/K$ is defined as the torsion group
$\tor(L^{\ast}/K^{\ast})$ (see \cite{greither}). The analogue
for Drinfeld modules we are interested in is obtained by
replacing the multiplicative structure of the field by the one given by
the Drinfeld module structure. We see that when $\rho$ is a
rank one $A$--Drinfeld module where $A$ is 
of class number one, the results of \cite{sanchez} can be
obtained also in this case. However, we will see that there is
no cogalois theory for arbitrary $A$--Drinfeld modules of 
rank one.

\section{Preliminaries and notations}\label{S2}

We consider function fields $K/{\mathbb{F}}_{q}$ where
we fix a prime divisor denoted by $\pp$. $A$ denotes
the Dedekind
ring consisting of the elements $u\in K$ such that
the only possible pole of $u$ is $\pp$.

We will use the following notation along the paper.

\begin{dinglist}{253}

\item $k={\ma F}_q(T)$ denotes the rational function field
over a finite field of $q$ elements ${\ma F}_q$.

\item $R_T={\ma F}_q[T]$ denotes the polynomial ring
over $T$ and such that $k$ is the quotient field of $R_T$.

\item $K$ is a global function field over ${\ma F}_q$.

\item ${\mathrm C}$ denotes the Carlitz module.

\item $\Lambda_M=\{u\in \bar{k}\mid 
{\mathrm C}_M(u)=u^M=0\}$ with $M\in R_T$.

\item $\pp$ is a fixed place of $K$ called the {\em infinite prime}  of $K$.

\item $d_{\infty}=\deg \pp$ denotes the degree of $\pp$.

\item $A=\{x\in K\mid v_{\eu p}(x)\geq 0\text{\ for every place\ }
{\eu p}\neq \pp\}$.

\item $h_K=|C_{K,0}|$ denotes the class number of $K$.

\item $h_A=d_{\infty}h_K$ is the class number of the Dedekind ring $A$.

\item ${\ma C}_{\infty}={\ma C}_p$ denotes the completion of an algebraic
closure of the completion $K_{\infty}$ of $K$ at $\pp$.

\item $\rho\colon A\to E\langle \tau\rangle$ is an $A$--Drinfeld module
of generic characteristic defined over a field extension $E$ of the
field of definition $K_{\rho}$ of $\rho$.

\item $\rho[I]=\{u\in \bar{K}\mid \rho_c(u)=0\text{\ for all\ } c\in I\}$
where $I$ is an ideal of $A$.

\item $\rho[a]=\rho[(a)]$ for $a\in A$.

\item For a nonzero ideal ${\eu m}$ of $A$ we let $\Phi({\eu m})=
\big|\big(A/{\eu m}\big)^{\ast}\big|$.

\item $\mu_{\rho}(L)=\mu(L)=\{u\in L\mid \rho_a(u)=0 \text{\ for\ }
a\in A\setminus\{0\}\}$ denotes the torsion of a Drinfeld module
of an extension $L$ of $K$.

\end{dinglist}

For the particular case $h_A=1$, necessarily we have
$d_{\infty}=\deg \pp=1$ and $h_{K}=1$.
Therefore there exist only $5$ such fields and rings $A$
according to the classification of function fields with
class number one. In this situation, we may and we will
assume that $K_{\rho}=K$.
We also ask whether the structural map of the Drinfeld module
$\rho$, $\delta:A\rightarrow E$, is the natural embedding. 
The Drinfeld modules under consideration will be of rank one,
unless otherwise specified. So, we have $\deg (\rho_{a})=
-d_{\infty}v_{\pp}(a)=\deg a$.

We have that $|A/(a)|=q^{\deg a}$ is finite, and
$\deg a=\dim_{{\ma F}_q} A/(a)$. 
If necessary, we will assume for a rank one
Drinfeld module $\rho$ that $K_{\rho}=H_A$ is the
Hilbert class field (see \cite[\S 15]{Hay92}). Let
$H_A^+$ be the {\em normalizer field} for $A$--Drinfeld
modules over $K$, $\pp$ for a fixed sign function $\sgn$.
We have that $H_A^+$ is the narrow Hilbert class field
with respect to $\sgn$. We know that $H_A^+/K$
is an abelian extension with Galois group isomorphic to
$\Pic^+ A=M_A/P_A^+$ where $M_A$ is the group 
of fractional ideals of $A$ and $P_A^+=\{xA\mid x\in K^{\ast},
\sgn(x)=1\}$. We have $|\Pic_A^+|=[H_A^+:K]=\frac{q^{d_{\infty}}-1}
{q-1}h_A$ (see \cite[Theorem 14.7]{Hay92}, 
\cite[Theorem 13.5.30]{Vil2006}).

If $\rho\colon A\to {\ma C}_{\infty}\langle\tau\rangle$ is a 
rank one $A$--Drinfeld module, then $\rho=\rho^{\Gamma}$
where $\Gamma=A\bar{\pi}$ is a lattice with
$\bar{\pi}\in {\ma C}_{\infty}\setminus \{0\}$ 
and the exponential function associated to $\rho$
is given by $\ex_{\Gamma}(x)=x\prod_{\gamma\in\Gamma\setminus\{0\}}
\big(1-\frac{x}{\gamma}\big)$. Thus $\ex_{\Gamma}(\gamma)=0$
if and only if $\gamma\in \Gamma$. We define $\lambda_a:=\ex_{\Gamma}
\big(\frac{\bar{\pi}}{a}\big)$ for $a\in A\setminus\{0\}$. 
From the functional equation $\ex_{\Gamma}(au)=\rho_a(\ex_{\Gamma}(u))$
we obtain that $\rho_a(\lambda_a)=0$. Further $\rho_m(\lambda_{mn})=
\lambda_n$ for $n,m\in A$.

\begin{observacion}\label{O3.9}
We have that $\lambda_a$ is a generator of the
$A$--module $\rho[a]$.
\end{observacion}

\section{General results on Drinfeld modules}\label{S3}

The following results will be used along the paper.

\begin{proposition}\label{P3.1}
Let $\rho$ be an $A$--Drinfeld module of rank one and let
$a\in A$ be nonzero. Then $K(\rho[a])/K$ is an abelian
extension and $\Gal(K(\rho[a])/K)$ is isomorphic to a 
subgroup of the group $(A/(a))^{*}$. \fin
\end{proposition}

\begin{proposition}\label{P3.2} Let $\eu P$ be a nonzero
prime ideal of $A$. Let $m\in{\ma N}$ and $K({\eu P}^m):=
H_A^+(\rho[{\eu P}^m])$.
Then the extension $K({\eu P}^m)=H_A^+(\rho[{\eu P}^m])/
H_A^+$ is totally ramified in ${\eu F}$, where ${\eu F}$
is the prime divisor of $H_A^+$ above ${\eu P}$
and the ramification index is equal to $\Phi({\eu P}^m)$.
Furthermore, the extension $K({\eu P}^m)/H_A^+$
is unramified at every prime divisor other than 
${\eu P}$ and the primes above $\pp$. We also
have $[K({\eu P}^m):H_A^+]= \Phi({\eu P}^m)$.

Finally, $\pp$ decomposes fully in $H_A/K$
and is totally ramified in $H_A^+/H_A$.
\end{proposition}

\begin{proof}
See \cite[Proposition 14.4, Theorem 15.6]{Hay92},
\cite[Proposition 13.5.41, Theorem 13.5.35]{Vil2006}.
\end{proof}

\begin{corollary}\label{C3.4}
For any nonzero ideal $\eu m$ of $A$, $K({\eu m}):=
H_A^+(\rho[{\eu m}])/H_A^+$ is a Galois extension with
Galois group isomorphic to $\big(A/{\eu m}\big)^{\ast}$.
The ramified finite primes are precisely the prime
ideals $\eu P$ dividing $\eu m$ with ramification index
equal to $\Phi({\eu P}^e)$ where ${\eu P}^e$ 
is the exact power of $\eu P$ that divides $\eu m$.
\end{corollary}

\begin{proof}
\cite[\S 16]{Hay92}, \cite[Corollary 13.5.42]{Vil2006}.
\end{proof}

\begin{teorema}\label{T3.6} If $A$ is arbitrary and $\rho$ 
is any $A$--Drinfeld module, then $\mu_{\rho}(L)$ is a
finite set for any finite extension $L$ of $K$.
\end{teorema}

\begin{proof} \cite{denis, GhiHsi2013, Poo97, Schweizer}.
\end{proof}

\begin{observacion}\label{O3.7}
The first proof of Theorem \ref{T3.6} was given by
Denis in \cite[Th\'eor\`eme 1]{denis}, where he proves
that the number of elements with height bounded
by a fixed real number $D$, is finite and that the torsion
elements are precisely the elements of height $0$.
To show only that the torsion is finite, the proof can be
reduced to the case $A=R_T$ (\cite[Proposition 1]{Poo97},
\cite[Remark 2.8]{GhiHsi2013}) as follows: if $\rho\colon A\to E\langle
\tau\rangle$ is an $A$--Drinfeld module over $E$, 
we choose $T\in K$ such that the pole divisor of $T$ is $\pp^n$
for some $n\geq 1$. Then $A$ is the integral closure of $R_T$ in $K$
and $\rho^{\prime}=\rho\circ i\colon R_T\to E\langle\tau\rangle$ is
an $R_T$--Drinfeld module over $E$ where $i$ is the natural
embedding. Then $\mu_{\rho^{\prime}}(L)=\mu_{\rho}(L)$, being
the fact $\mu_{\rho^{\prime}}(L)\subseteq \mu_{\rho}(L)$
clear. Now, if $x\in \mu_{\rho}(L)$, then $\rho_a(x)=0$
with $a\in A\setminus\{0\}$. Consider 
$\alpha_0,\ldots,\alpha_{n-1}\in R_T$, $\alpha_0\neq 0$
and $\alpha_0+\alpha_1 a+\cdots+\alpha_{n-1}a^{n-1}+a^n=0$,
so $\rho_{\alpha_0}(x)=0$ and thus $x\in\mu_{\rho^{\prime}}(L)$.

Once we have reduced to the case $R_T$, the proof of the finiteness
of $\mu_{\rho}(L)$ can be obtained over a finite extension $L$ of
$K_{\eu p}$, where $K_{\eu p}$  is the completion of $K$ at
${\eu p}\neq \pp$ proving that if $x\in\mu_{\rho}(L)$, then
$v_{\eu p}(x)\geq c$ for some $c$ and therefore
$\overline{\mu_{\rho}(L)}$ is a compact discrete set, hence
finite (\cite[Proposition 1]{Poo97}).

In the particular case of rank one over $R_T$ and $A$ with
$h_A=1$, we have $\mu_{\rho}(L)=\rho[a]$ for some $a\in
A\setminus \{0\}$. In particular, $\mu_{\mathrm C}(L)=\Lambda_M$ for
some $M\in R_T \setminus\{0\}$ in the Carlitz module case.
This is a very particular case of the general case of Drinfeld
modules of rank one.
\end{observacion}

\begin{proposition}\label{P3.8} Let $L/E$ be a finite extension
and let $\rho$ be a Drinfeld module of rank one. Then there
exists an ideal $\mathfrak{b}$ of $A$ such that
$\mu_{\rho}(L)=\rho[\mathfrak{b}]$. In particular, if $\rho={\mathrm C}$
is the Carlitz module, $\mu_{\mathrm C}(L)=\Lambda_M$ for some
$M\in R_T$.
\end{proposition}

\begin{proof}
Let $x\in\mu_{\rho}(L)$ and choose $a\in A$, $a\neq 0$ such that
$\rho_a(x)=0$. Consider the annihilator of $x$: $\an(x):=\{
b\in A\mid \rho_b(x)=0\}={\eu a}$. Then ${\eu a}$ is a nonzero
ideal of $A$. Let $M:=A\cdot x=\{\rho_b(x)\mid b\in A\}$.
We have that $M\subseteq L$ since $L$ is an $A$--module and
$M\cong A/{\eu a}$. On the other hand $\rho[{\eu a}]\cong A/{\eu a}$
because $\rho$ is of rank one.
Clearly we have that $M\subseteq \rho[{\eu a}]$ and since
both sets are of the same cardinality $|A/{\eu a}|$,
it follows that $M=\rho[{\eu a}]$. Therefore $\rho[{\eu a}]\subseteq
\mu_{\rho}(L)$.

Since $\mu_{\rho}(L)$ is finite, we write $\mu_{\rho}(L)=
\{x_1,\ldots,x_m\}$ and let ${\eu a}_i:=\an(x_i)$, $1\leq i\leq m$.
Let ${\eu b}=\mcm\{{\eu a}_i\mid 1\leq i\leq m\}=\cap_{i=1}^m
{\eu a}_i$. Let $\eu c_i$ be an integral ideal such that ${\eu b}={\eu a}_i
{\eu c}_i$, $1\leq i\leq m$. We have $\rho_b(x_i)=0$ for all
$b\in{\eu b}$ and for all $1\leq i\leq m$. It follows that
$\mu_{\rho}(L)\subseteq \rho[{\eu b}]$.

Let now $x\in\rho[{\eu b}]$. We claim that $\{{\eu c}_i\}_{i=1}^m$
are relatively primes, that is, $A=\sum_{i=1}^m {\eu c}_i=\mcm\{
{\eu c}_i\mid 1\leq i\leq m\}$. Otherwise, we would have a nonzero
prime ideal $\eu p$ dividing $\eu c_i$ for all $1\leq i\leq m$.
Let ${\eu d}_i$ be such that ${\eu c}_i={\eu d}_i
{\eu p}$, hence ${\eu b}={\eu a}_i{\eu c}_i={\eu a}_i{\eu d}_i
{\eu p}$ and therefore ${\eu b}{\eu p}^{-1}={\eu a}_i{\eu d}_i$ 
so it follows that ${\eu a}_i$ divides ${\eu b}{\eu p}^{-1}$
for all $i$ from it would follow ${\eu b}\mid {\eu b}{\eu p}^{-1}$
which is absurd. This shows that $\mcm\{{\eu c}_i\mid
1\leq i\leq m\} =A$.

Let $c_i\in{\eu c}_i$, $1\leq i\leq m$, be such that 
$1=\sum_{i=1}^m c_i$. For all $1\leq i\leq m$ and for all
$d\in{\eu a}_i$, we have $dc_i\in{\eu a}_i{\eu c}_i={\eu b}$
so that
\begin{gather*}
\rho_{dc_i}(x)=\rho_d(\rho_{c_i}(x))=\rho_{c_i}(\rho_d(x))=0,\\
\intertext{which implies $\rho_{c_i}(x)\in \rho[{\eu a}_i]\subseteq
\mu_{\rho}(L)$. Therefore}
x=\rho_1(x)=\sum_{i=1}^m \rho_{c_i}(x)\in \mu_{\rho}(L),
\end{gather*}
thus we obtain $\rho[{\eu b}]\subseteq \mu_{\rho}(L)$
and hence $\rho[{\eu b}]=\mu_{\rho}(L)$.
\end{proof}

\begin{proposition}\label{P3.5} 
Let $\rho$ be an $A$--Drinfeld module of rank one over 
$E\subseteq H_A^+$.
Let $q>2$ and $\eu m$ a nonzero ideal of $A$. Consider
the extension $K({\eu m})=H_A^+(\rho[{\eu m}])/H_A^+$.
Then $\mu_{\rho}(K({\eu m})) = \rho[{\eu m}]$.
\end{proposition}

\begin{proof} It suffices to show that $\mu_{\rho}(K({\eu m}))
\subseteq \rho[{\eu m}]$.
If $u\in \mu_{\rho}(K({\eu m}))$, then $\rho_a(u)=0$ for
some $a\in A\setminus\{0\}$. Let $\an_{\rho}(u)={\eu c}\neq 0$.
From Proposition \ref{P3.8} it follows that $\rho[{\eu c}]\subseteq
K({\eu m})$, $K(\rho[{\eu c}])\subseteq K(\rho[{\eu m}])$ and
$K({\eu c})\subseteq K({\eu m})$. From Proposition
\ref{P3.2} it follows, analyzing the ramification
index of each prime in the extensions $K({\eu c})/H_A^+$ and
$K({\eu m})/H_A^+$, that ${\eu c}\mid {\eu m}$ and therefore
$u\in \rho[{\eu c}]\subseteq \rho[{\eu m}]$.
\end{proof}

\section{Radical extensions}\label{S5}

Let $K/{\mathbb{F}}_{q}$ be a function field and let $\pp$
be a fixed prime divisor. Let us consider
$\delta:A\rightarrow E$ the natural embedding and
$\rho:A\rightarrow E\langle\tau\rangle$ a rank one Drinfeld module.
As before, we consider function field extensions $L/K$ 
such that $E\subseteq K\subseteq L\subseteq \overline{k}$.
Since this type of extensions are $E$--algebras, we may
give them an $A$--module structure using the map $\rho$.
The first objet to consider, associated to the extension
$L/K$, is the following:
\[
\DrinT(L/K) = \{u\in L\mid \text{there exists $m\in
A\setminus \{0\}$ such that $\rho_{m}(u) \in K$}\}.
\]

Note that $\DrinT(L/K)\subseteq L$ is a subgroup of the
additive group $L$. On the other hand, $\DrinT(L/K)$ is
an $A$--module and the $A$--module $\DrinT(L/K)/K$
is of $A$--torsion. The module $\cogaloistor(L/K)= \DrinT(L/K)/K$ will be called
the {\em Drinfeld cogalois module}\index{Drinfeld
cogalois module} of the extension $L/K$.
We have that $\cogaloistor(L/K)$
is analogous to the group $T(L/K)/K^{*}$ in a field extension
$L/K$ where $T(L/K)$ denotes the usual torsion group
associated to the extension $T(L/K)$, that is,
$T(L/K)=\{u\in L\mid \text{there exists $n\in {\mathbb{N}}$
such that $u^{n}\in K$}\}$ (see \cite{greither}).

\begin{definicion}\label{definicion1_c2drinfelduno}
We say that an extension $L/K$ is {\em radical}\index{radical
extension} if there exists a set $X\subseteq \DrinT(L/K)$ 
such that $L= K(X)$. We will say that $L/K$ is 
{\em pure}\index{pure extension} if for any irreducible
$m\in A$ and for each $\lambda_{m}\in L$, we have
$\lambda_{m}\in K$. Finally we will say that $L/K$ is
a {\em cyclotomic coradical extension}\index{cyclotomic
coradical extension} if it is radical, separable and pure.

\end{definicion}

\subsection{Reduction to the case $R_T$}\label{S4}

In this subsetion we show how we may reduce the general case to
the case $A=R_T$.

Let $A$ and $K$ be arbitrary. Let $T\in A$ be such that 
${\mathfrak N}_T={\mathfrak p}_\infty^n$, for some $n\geq 1$.
Consider $R_T={\ma F}_q[T]$ and $k={\ma F}_q(T)$.
Let $\rho\colon A\to E\langle\tau\rangle$ be a $A$--Drinfeld 
module of rank $r_{\rho}$. Let $\iota\colon R_T\to A$
be the natural embedding and let $\rho^{\prime}\colon R_T\to E
\langle\tau\rangle$ be given by $\rho^{\prime}=\rho\circ \iota$.
Let $r_{\rho^{\prime}}$ be the rank of $\rho^{\prime}$.
Then

\begin{proposition}\label{P3.1'}
$r_{\rho^{\prime}}=d_{\infty}nr_{\rho}$.
\end{proposition}

\begin{proof} We have
$\deg_{R_T}\alpha_0=-v_{\p}(\alpha_0)$, where $\p$ 
is the pole of $T$. Further, we have $d_{\p}=1$ and
\[
\deg_{A}\alpha_0=
-d_{\infty}v_{\pp}(\alpha_0)=-d_{\infty}e({\mathfrak p}_{\infty}|
\p)v_{\p}(\alpha_0)=d_{\infty}n\deg_{R_T} \alpha_0.
\]
It follows that
\begin{align*}
\deg\rho_{\alpha_0}&=r_{\rho^{\prime}}\deg_{R_T}{\alpha_0}
\quad \text{and}\\
\deg\rho_{\alpha_0}&=r_{\rho}\deg_A \alpha_0=
d_{\infty}n\deg_{R_T} \alpha_0.
\end{align*}
Therefore $r_{\rho^{\prime}} =d_{\infty}nr_{\rho}$.
\end{proof}

\begin{proposition}\label{P3.1''}
Let $L/E$ be a finite extension. Then with the conditions
of Proposition {\rm{\ref{P3.1'}}}, we have
$\cog_{\rho}(L/E)=\cog_{\rho^{\prime}}(L/E).$
\end{proposition}

\begin{proof}
Let $x\in{\cog}_{\rho^{\prime}}(L/E)$. There exists
$\alpha_0\in R_T$, $\alpha_0\neq 0$, such that
$\rho'_{\alpha_0}(x)\in E$. Therefore
$\rho^{\prime}_{\alpha_0}(x)=\rho_{\iota \alpha_0}(x)=\rho(\iota
\alpha_0)(x)
=\rho(\alpha_0)(x)=\rho_{\alpha_0}(x)\in E$. It follows that
$x\in{\cog}(L/E)$, and therefore $\cog_{\rho^{\prime}}
(L/E)\subseteq{\cog}_{\rho}(L/E)$.

Now let $x\in{\cog}_{\rho}(L/E)$. There exists $a\in A,$ $a\neq 0$ such that
$\rho_{a}(x)\in E$. Since $A$ is the integral closure of $R_T$ in $K$,
there exist $\alpha_0\neq 0,\alpha_1,\ldots,\alpha_{m-1}\in R_T$
such that $\alpha_0+\alpha_1a+\cdots+\alpha_{m-1}a^{m-1}+a^m=0$.
Therefore, considering $\alpha_{m}=1$, we obtain
\[
0=\rho_{0}(x)=\rho_{\sum\limits_{i=0}^{m}
\alpha_{i}a^{i}}(x)=\sum\limits_{i=0}^{m}\rho_{\alpha_{i}}
\rho_{a^i}(x)=\rho_{\alpha_{0}}(x)+
\sum\limits_{i=0}^{m}\rho_{\alpha_{i}}(\rho_{a}
\circ\cdots\circ\rho_{a})(x)\in E
\]
which implies $\rho_{\alpha_{0}}(x)\in 
E$. Hence $x\in{\cog}_{\rho'}(L/E)$. This finishes the proof.
\end{proof}

Now we return to the general case.

In the following examples, we consider the field ${\mathbb{F}}_{q}$
with $q>2$.

\begin{ejemplo}\label{ejemplo1drinfeld}
Let $A=R_T$ and $E=k$. The extension $k(\rho[m])/k$,
with $m\in A$ non constant, is radical since there exists
$W = \rho[m]\subseteq \DrinT(k(\rho[m])/k)$ 
such that $k(\rho[m]) = k(W)$. It is also a separable extension
but it is not pure since from Proposition \ref{P3.5}, we have
that $\mu(k(\rho[m]))$ is equal to $\rho[m]$ and if $c$
is an irreducible factor of $m$,
$\lambda_{c}\in \rho[m]$ does not belong to $k$.
Therefore $k(\rho[m])/k$ is not a cyclotomic coradical
extension.
\end{ejemplo}

\begin{ejemplo}\label{ejemplo5Drinfeld}
Let $A=R_T$ and $E=k$. Let $c\in A$ be an irreducible
polynomial. The extension $k(\rho[c^{n}])/k(\rho[c])$
is cyclotomic coradical since it is clear that is radical and
separable because the polynomial with coefficients 
in $k$, $\rho_{c^{n}}(U)$, is separable.
On the other hand, let $d\in A$ be an irreducible polynomial
in such a way that $\lambda_{d}\in k(\rho[c^{n}])$. From
Proposition \ref{P3.5}, we have that $d\mid c^n$
and thus $d\mid c$. It follows that the extension is pure.

The same argument may be applied for arbitrary $A$,
$\rho$ of rank one, $\eu c$ a nonzero ideal of $A$ and
the extension $K({\eu c}^n)/K({\eu c})$ for $n\in{\ma N}$.
\end{ejemplo}

Let us assume that $L/K$ is a radical extension. Therefore
there exist $\alpha_{i}\in L$ and $a_{i}\in A$ such that
$\rho_{a_{i}}(\alpha_{i})=\beta_{i}\in K$ and $L=K(\{\alpha_{i}\})$.
Now we take $\alpha_{i}$ arbitrary. We denote such element
only by $\alpha$. We define
\[
\varphi_{\overline{\alpha}}:A\rightarrow\cogaloistor(L/K)
\]
for $\varphi_{\overline{\alpha}}(a)=\overline{\rho_{a}(\alpha)}$.

This map is well defined and $I=\ker (\varphi_{\overline{\alpha}})$
is an nonzero ideal of $A$ distinct from $A$ itself. Hence, in case
$h_A=1$, there exists $a\in A$ such that $I=(a)$. We say that
$\alpha$ is of {\em order $a$}. This definition is ambiguous since
other generator of $I$ might be used as the order. However
we will accept this ambiguity.

\begin{teorema}\label{T5.1} Let $h_A=1$ and $E=K$.
Let $L/K$ be a radical extension, say $L=K(\{\alpha_{i}\})$.
Then $L/K$ is a Galois extension if and only for each
$\alpha_{i}$ of order $a_{i}$, we have $\lambda_{a_{i}}\in L$.
\end{teorema}

\begin{proof}
We first consider the case $L/K$ is a Galois extension.
Let $\alpha=\alpha_i$ be of order $a=a_i$.  Consider the
polynomial $f(U)=\rho_{a}(U)-\beta\in K[U]$, where
$\beta=\rho_{a}(\alpha)$. Now
$f(U)=\prod(U-(\alpha+\rho_{b}(\lambda_{a})))$, so that
$\Irr(U,\alpha,K)$ divides $f(U)$. It follows that the conjugates
of $\alpha$ in $L$ are of type
\[
\{\alpha+\xi_{1},\ldots,\alpha+\xi_{s}\},
\]
with $\xi_{j}\in \rho[a]$. Note that this set is contained in $L$.

Let $\xi_{i}=\rho_{b_{i}}(\lambda_{a})$ for some $b_{i}\in A$
and let $B$ be the $A$--module generated by $\{\xi_{1},\ldots,\xi_{s}\}
\subseteq\rho[a]$. There exists $a^{\prime}\in A$ such that
$a^{\prime}\mid a$ and $B=\rho[a^{\prime}]$.
Let $\beta^{\prime}=\rho_{a^{\prime}}(\alpha)$. Consider the
polynomial $g(U)=\rho_{a^{\prime}}(U)-\beta^{\prime}\in
K[U]$. We have $\Irr(U,\alpha,K)\mid g(U)$ and $g(\alpha)=0$.
Therefore $a^{\prime}\in I= \ker(\varphi_{\overline{\alpha}})$.
Hence $a\mid a^{\prime}$ and it follows that 
$a^{\prime}=au$ with $u\in A$ a unit. So, $\lambda_{a}\in L$.

Conversely, let $a$ be the order of $\overline{\alpha}$. Every
conjugate of $\alpha$ is of the form $\alpha+\rho_{b}
(\lambda_{a})\in L$. Thus $L/K$ is a normal extension. 
Since $\alpha$ is separable over $K$, it follows that the
extension $L/K$ is a Galois extension.
\end{proof}

\begin{proposition}\label{primitivo_explicito}
Let $L/E$ be an extension such that $L=K(\alpha,\beta)$
and such that there exist $m,n\in A$ with
$\rho_{m}(\alpha)\in K$ and $\rho_{n}(\beta)\in K$
being $m$ and $n$ relatively prime, that is, $(m)+(n)=A$.
Then $L=K(\alpha+\beta)$, that is, $\alpha+\beta$ is
a primitive element for the extension and it also belongs
to $\DrinT(L/K)$.
\end{proposition}

\begin{proof}
We have $K(\alpha+\beta)\subseteq K(\alpha,\beta)$. Let $\xi_{1}=
\rho_{m}(\alpha)$ and $\xi_{2}= \rho_{n}(\beta)$. Then
$\rho_{m}(\alpha+\beta)= \rho_{m}(\alpha)+\rho_{m}(\beta)=
\xi_{1}+\rho_{m}(\beta)\in K(\alpha+\beta)$ and
$\rho_{n}(\alpha+\beta)= \rho_{n}(\alpha)+\rho_{n}(\beta)=
\rho_{n}(\alpha)+\xi_{2}\in K(\alpha+\beta)$. Thus
$\rho_{m}(\beta),\rho_{n}(\alpha)\in K(\alpha+\beta)$.

Let $s_1,s_2\in A$ be such that $1=ms_{1}+ns_{2}$. Then
\begin{gather*}
\alpha=\rho_{1}(\alpha)=\rho_{ms_{1}+ns_{2}}(\alpha)=
\rho_{s_{1}}(\xi_{1})+\rho_{s_{2}}(\rho_{n}(\alpha))\in K(\alpha+\beta)
\intertext{and}
\beta=\rho_{1}(\beta)=\rho_{ms_{1}+ns_{2}}(\beta)= \rho_{s_{1}}(\rho_{n}(\beta))+ \rho_{s_{2}}(b)\in K(\alpha+\beta),
\intertext{so that $K(\alpha,\beta)= K(\alpha+\beta)$. Furthermore}
\rho_{mn}(\alpha+\beta)= \rho_{n}(\rho_{m}(\alpha))+\rho_{m}(\rho_{n}(\beta))\in K.
\end{gather*}
\end{proof}

We observe that the result can be generalized to extensions
$L/E$ such that $L=E(\alpha_{1},\ldots,\alpha_{s})$ and
such that there exist $m_{i}\in A$ with 
$\rho_{m_{i}}(\alpha_{i})= \beta_{i}\in E$ and the elements
$m_i$ are pairwise relatively prime.

Next we will give some definitions which are analogous to the ones given
in \cite{Schultheis}. Here, we consider $A$ of class number one
and $E=K$. As before, let $\rho$ denote a Drinfeld module. Let
$a\in A\setminus \{0\}$, $K$ be any finite extension of 
$k(\rho[a])$ and $z\in K\setminus
K_{a}$, where $K_{a}=\{\rho_{a}(y)\mid y\in K\}$. 
Consider the polynomial $G(U)=\rho_{a}(U)-z$. The decomposition
field of $G(U)$ will be called a {\em Drinfeld--Kummer extension}.
Multiplying $a$ by a suitable constant, we may assume that
$G(U)$ is a monic polynomial. This type of extensions will be
denoted by $K_{a,z}$. On the other hand, note that in general
the polynomial $G(U)$ is not irreducible over $K$. Let
$G_{1}(U),\ldots,G_{s}(U)$ be the irreducible monic factors of $G(U)$.

Next proposition establishes some properties of these extensions.

\begin{proposition}\label{proposicionDrinfeldKummerext}
Let $K$ be a finite extension of $k(\rho[a])$. Let $z\in K\setminus
K_{a}$ and let $K_{a,z}$ be the associated Drinfeld--Kummer
extension. Then
\las
\item $G(U)$ is a separable polynomial of degree
$q^{m}$ with $m=\deg a$.

\item If $\alpha\in \overline{k}$ is any root of $G(U)$, then
$W=\{\alpha+\lambda\mid \lambda\in \rho[a]\}$ is the set
of all roots of $G(U)$ and $K_{a,z}=K(\alpha)$.

\item There exists $s\in {\mathbb{N}}$ such that $[K_{a,z}:K]=p^{s}$.
\end{list}
\end{proposition}

\begin{proof}
(1) By the conditions satisfied by $\rho$, it is clear that $G(U)$
is separable and of the claimed degree.

(2) Since $\rho$ is a linear map it follows that $W$ is the
set of all roots of $G(U)$.

(3) Consider the Galois group $\Gal(L/K)$ of the extension
$L/K$. Consider $\sigma\in \Gal(L/K)$. Then
$\sigma(\alpha)=\alpha+\lambda_{\sigma}$ with $\lambda_{\sigma}\in
\rho[a]$. We define $\Theta:\Gal(L/K)\rightarrow
\rho[a]$ given by $\Theta(\sigma)=\lambda_{\sigma}$. We
have that $\Theta$ is well defined and since
$\sigma(\tau(\alpha))=\sigma(\alpha+\lambda_{\tau})=
\alpha+\lambda_{\sigma}+\lambda_{\tau}$,
it follows that $\Theta$ is a group homomorphism. 
Finally, if $\Theta(\sigma)=0$, then $\lambda_{\sigma}=0$,
that is, it follows from the definition of $\Theta$, 
that the automorphism
$\sigma$ is the identity map. Therefore $\Theta$ is
a group monomorphism. In particular $\Gal(L/K)$ is
an elementary $p$--abelian group, that is, for each 
$\sigma\in \Gal(L/K)$ we have $p\sigma=1$. The result follows.
\end{proof}

\section{Cyclotomic coradical extensions}\label{S6}

Cyclotomic coradical extensionshave several 
properties analogous to the ones of classical cogalois
extensions.

\begin{lema}\label{lemapurezadrinfeld}
Let $E\subseteq L\subseteq L^{\prime}$ be a tower of fields. Then
$L^{\prime}/E$ is pure if and only if $L^{\prime}/L$ and $L/E$ are
pure.
\end{lema}

\begin{proof}
Analogous to \cite[Lemma 5.1]{sanchez}.
\end{proof}

\begin{proposition}\label{proposicionprop_cogdrinfeld}
Let  $E\subseteq L\subseteq L^{\prime}$ be a tower of fields. Then
\las
\item There is an exact sequence of $A$--modules
\[
0\rightarrow \cogaloistor(L/E)\rightarrow 
\cogaloistor(L^{\prime}/E)\rightarrow \cogaloistor(L^{\prime}/L).
\]
\item  If the extension $L^{\prime}/E$ is cyclotomic coradical, 
then the extension $L^{\prime}/L$ is cyclotomic coradical.

\item If the extension $L^{\prime}/E$ is radical and the extensions
$L^{\prime}/L$ and $L/E$ are cyclotomic coradical, then 
$L^{\prime}/E$ is cyclotomic coradical.
\end{list}
\end{proposition}

\begin{proof}
Similar to \cite[Proposition 5.2]{sanchez}.
\end{proof}

Note that if $L/E$ is a Galois field extension, then 
$\mu_{\rho}(L)$ is a $G=\Gal(L/E)$--module with the natural action.

Next theorem holds for any finite extension and any $A$--Drinfeld
module $\rho$.

\begin{teorema}\label{teoremafinitud_TC/Ldrinfeld}
Let $L/E$ be a finite Galois extension and let $G$ be its
Galois group. Then the map $\phi\colon \cogaloistor(L/E)\rightarrow
Z^{1}(G,\mu(L))$ given by $\phi(u + E) = f_{u}$ where $f_{u}(\sigma)
= \sigma(u) - u$, is a group isomorphism.
\end{teorema}

\begin{proof}
Analogous to \cite[Theorem 5.4]{sanchez}.
\end{proof}

\begin{corollary}\label{cor_finitud_TC/Ldrinfeld}
Let $L/E$ be a finite extension. Then the $A$--module
$\cogaloistor(L/E)$ is finite.
\end{corollary}

\begin{proof}
We take the Galois closure $\tilde{L}/E$ of $L/E$. The
result follows from Proposition \ref{proposicionprop_cogdrinfeld},
Theorem \ref{teoremafinitud_TC/Ldrinfeld} and from Theorem
\ref{T3.6}.
\end{proof}

Several results from \cite{sanchez} also hold in our situation.

From now on, the following results only hold for $A$ such
that $h_A=1$.

\begin{proposition}\label{cogalois_disjuntodrinfeld}
Let $A$ be such that $h_A=1$ and $L/E$ be a field
extension such that
\[
[L:E]=\ell
\]
with $\ell$ a prime number different from $p = \car k$. Then $L/E$
is not a cyclotomic coradical extension.
\end{proposition}

\begin{proof}
Analogous to \cite[Proposition 6.1]{sanchez}.
\end{proof}

\begin{corollary}\label{cogalois_disjuntocor1drinfeld}
Let $A$ with $h_A=1$ and let $L/E$ be a Galois extension such that
\[
[L:E]= p^{s}n
\]
with $p\nmid n$ and $n>1$. Then $L/E$ is not a cyclotomic
coradical extension.
\end{corollary}

\begin{proof}
Analogous to \cite[Corollary 6.2]{sanchez}.
\end{proof}

\begin{corollary}\label{estructura1drinfeld}
If $h_A=1$ and $L/E$ is a cyclotomic coradical Galois
extension, then $[L:E]$ is of the form $p^{s}$, with $s\in {\mathbb{N}}$.
\fin
\end{corollary}

\begin{lema}\label{pureza_pdrinfeld}
If $h_A=1$ and if $L/E$ is an extension such that $[L:E] = p^{s}$
with $s\in {\mathbb{N}}$, then $L/E$ is pure.
\end{lema}

\begin{proof}
Analogous to \cite[Lemma 6.4]{sanchez}.
\end{proof}

As a consequence of the previous results, we obtain

\begin{teorema}\label{tdim_p_galoisdrinfeld} 
Assume $h_A=1$. A Galois extension $L/E$ is cyclotomic coradical
if and only if it is radical, separable and $[L:E]=p^{s}$ for some $s\in
{\mathbb{N}}$. \fin
\end{teorema}

\begin{teorema}\label{lemadrinfeldextsimcogalois} 
Assume $h_A=1$ and let $L/E$ be a pure extension. Assume that
$L=E(\alpha)$ and that there exists an irreducible element
$d\in A$ such that $\rho_{d}(\alpha)=x\in E$.
Then $L/E$ is a cyclotomic coradical extension and there
exists $s\in {\mathbb{N}}$ such that $[L:E]=p^{s}$.
\end{teorema}

\begin{proof}
Similar to \cite[Theorem 6.7]{sanchez}.
\end{proof}

As a consequence we obtain:

\begin{teorema}\label{tdim_pdrinfeld}
If $L/E$ is a cyclotomic coradical extension and if $h_A=1$, 
then $[L:E]=p^{n}$ for some $n\geq 0$.
\end{teorema}

\begin{proof}
Let $L/E$ be a cyclotomic coradical extension. Then
$L=E(\alpha_{1},\ldots,\alpha_{t})$, so that 
$\rho_{m_{i}}(\alpha_{i})=a_{i}\in E$ where
$m_{i}\in A$. Taking $m_{i}=d^{\varepsilon_{1,i}}_{1}\cdots
d^{\varepsilon_{i,i}}_{r_{i}}$, $\delta_{i,j}=
\rho_{\frac{m_{i}}{d_{i,j}}}(\alpha_{i})$ we have that
$\rho_{d_{i,j}}(\delta_{i,j})=a_{i}$. It follows that there exists a field
tower
\[
E\subseteq E(\beta_{1})\subseteq E(\beta_{1},\beta_{2})
\subseteq \cdots \subseteq E(\beta_{1},\ldots, \beta_{s})=L
\]
where for each $i=1,\ldots, s$ we have that
$\rho_{d_{i}}(\beta_{i})=b_{i}\in E(\beta_{1},\ldots ,\beta_{i-1})$ and
\begin{equation}\label{ec6_1drinfeld}
[L:E]=
\prod^{s}_{i=1}[E(\beta_{1},\ldots,\beta_{i}):E(\beta_{1},\ldots,
\beta_{i-1})].
\end{equation}

From equation (\ref{ec6_1drinfeld}) follows that it suffices to show
that if $L=E(\alpha)$ with $\rho_m (\alpha)=a\in E$ and if $m\in A$
is irreducible, and so $L/E$ is a cyclotomic coradical extension,
then $[L:E]=p^{i}$ for some $i\in
{\mathbb{N}}$. The later claim follows from Lemma
\ref{lemadrinfeldextsimcogalois}.
\end{proof}

\begin{corollary}\label{tdim_p_colateraldrinfeld}
With the notations of Theorem {\rm{\ref{tdim_pdrinfeld}}}
we have
\begin{gather*}
E(\alpha)\cap E(\lambda_{d})=E, \quad [L:E]=
[L(\lambda_{d}):E(\lambda_{d})]
\intertext{and}
\Irr(u,\alpha,E)=\Irr
(u,\alpha,E(\lambda_{d}))=F_{1}(u)=\prod(u-(\alpha+\lambda^{A}_{d})).
\end{gather*}
\end{corollary}

\begin{proof}
It follows from the proof of Theorem \ref{tdim_pdrinfeld}.
\end{proof}

Next corollary is analogous to Theorem
\ref{tdim_p_galoisdrinfeld} except that we do not assume that
the extension $L/E$ is Galois.

\begin{corollary}\label{principal_resultadodrinfeld} Assume $h_A=1$.
An extension $L/E$ is cyclotomic coradical if and only if it is
separable, radical and $[L:E]=p^{m}$ for some $m\in {\mathbb{N}}$.
\end{corollary}

\begin{proof}
It follows from Theorem \ref{tdim_pdrinfeld} and Lemma
\ref{pureza_pdrinfeld}.
\end{proof}

\section{Some computations}\label{S7}

Let $L/E$ be a finite cyclotomic coradical Galois extension. Therefore
we have that $L=E(\alpha_{1},\ldots,\alpha_{s})$ for some $\alpha_{i}\in L$ and
for each $\alpha_{i}$ there exists $a_{i}\in A$ such that
$\beta_{i}=\rho_{a_{i}}(\alpha_{i})\in E$. We may consider the
polynomials $f_{i}(U)=\rho_{a_{i}}(U)-\beta_{i}$. The set of roots
of each polynomial $f_{i}(U)$ is of the form
$\{\alpha_{i}+\rho_{c}(\lambda_{i})\}_{c\in A}$. It follows that
$\grugal(E(\alpha_{i})/E)\subseteq \rho[a_{i}]$.
Therefore $\grugal(E(\alpha_{i})/E)$ is an elementary $p$--abelian group.
Since we have the group monomorphism
\[
\grugal(L/E)\hookrightarrow \prod\grugal(E(\alpha_{i})/E)
\]
it follows that $\grugal(L/E)$ is an elementary $p$--abelian group.

\begin{lema}\label{auxiliar_coho_cogalois}
Let $L/E$ be a finite Galois cyclotomic coradical extension. Then
\[
B^{1}(G,\mu(L))\cong \mu(L)/\mu(E),
\]
where $G=\grugal(L/E)$.
\end{lema}

\begin{proof}
The map $\psi:\mu(L)\rightarrow B^{1}(G,\mu(L))$ is defined as follows:
$\psi(u)=f_{u}$, where $u\in \mu(L)$ and $f_{u}= \sigma(u)-u$ for
each $\sigma \in G$. It is clear that $\psi$ is a group isomorphism.
\end{proof}

Let $\mu(L) = \rho[a]$ for some $a\in A$. We define
\[
\deg(\mu(L))=\deg a.
\]

\begin{proposition}\label{proposiciontamanomoduloconraices}
Consider $A$ with $h_A=1$ and let $L/E$ be a Galois cyclotomic
coradical extension. Assume that $\mu(L)=\mu(E)$
and let $a\in A$ be such that $\mu(L)=\rho[a]$. Then
\[
|\cogaloistor(L/E)|=q^{m\deg(\mu(L))},
\]
where $m=|G|$ with $G=\grugal(L/E)\cong C^{m}_{p}$.
\end{proposition}

\begin{proof}
First note that $B^{1}(G,\mu(L))=\{0\}$. On the other hand, since the action
of $G$ is trivial on $\mu(L)$, we obtain $H^{1}(G,\mu(L))=\Hom(G,\mu(L))$.
From Theorem \ref{teoremafinitud_TC/Ldrinfeld} follows 
\[
\cogaloistor(L/E)\cong Z^{1}(G,\mu(L))/B^{1}
(G,\mu(L))\cong H^{1}(G,\mu(L))=\Hom(G,\mu(L)).
\]

Now $|\mu(L)|=C^{s\deg(\mu(L))}_{p}$, so that
\begin{align*}
\Hom(G,\mu(L)) &= \Hom(C^{m}_{p},C^{s\deg(\mu(L))}_{p})\\
&={\mathfrak{L}}_{p}({\mathbb{F}}^{m}_{p},{\mathbb{F}}^{s\deg(\mu(L))}_{p})
={\mathfrak{M}}_{m\times s\deg(\mu(L))}({\mathbb{F}}_{p}).
\end{align*}

Therefore $| \Hom(G,\mu(L))|=q^{m\grado(\mu(L))}$.
\end{proof}

For the following result, we consider $h_A=1$ and $E=K$.

\begin{teorema}\label{acotacion2}
Let $L/K$ be a Galois cyclotomic coradical extension and assume
that $L=K(\mu(L))$. Then
$|\cogalois(L/K)|\leq q^{m\grado(\mu(L))}$.
\end{teorema}

\begin{proof}
Similar to \cite[Proposition 8.5]{sanchez}.
\end{proof}

\section{Case $h_A>1$}\label{S8}

The fundamental results we have obtained for the cyclotomic
coradical extension with $A$ having class number one, are not
true any longer for $A$ with $h_A>1$. We give an example
showing why the results fail to hold.

Let $K={\ma F}_q(T)$ with $p=q=3$. Let $\pp$ be the place associated
to $T^2+1$ and let $A=
\{x\in K\mid v_{\eu p}(x)\geq \text{\ for every place ${\eu p}\neq \pp$}\}$.
Then
\[
A=\Big\{\frac{G(T)}{(T^2+1)^n}\mid G(T)\in{\ma F}_q[T], n\in{\ma N}, 
\deg G(T)\leq 2n\Big\}.
\]

Since $\pp$ is of degree $2$ and $h_K=1$, we have $h_A=2$.

Let $\xi=\frac{1}{T^2+1}$ and consider $R_{\xi}={\ma F}_q[\xi]$. Let
${\ma F}_q(\xi)$ denote the quotient field of 
$R_{\xi}$. Then $A$ is the integral closure of
$R_{\xi}$ in $K$. Using the division algorithm, it follows that if
$x\in A$, then $x=\frac{G(T)}{(T^2+1)^n}$ with $\deg G(T)\leq 2n$ and
\[
G(T)=\alpha_0+\alpha_1 (T^2+1)+\cdots +\alpha_n(T^2+1)^n
=\alpha_0+\alpha_1\xi^{-1}+\cdots +\alpha_n\xi^{-n},
\]
where $\alpha_i\in {\ma F}_q[T]$ is of degree less than or equal to $1$. 
Furthermore, because $\deg G(T)\leq n$, it follows that
$\alpha_n\in{\ma F}_q$.

Therefore
\begin{align*}
x=\frac{G(T)}{(T^2+1)^n}=\xi^n G(T)&=\alpha_n+\alpha_{n-1}\xi+\cdots
+\alpha_1\xi^{n-1}+\alpha_0\xi^n\\
&=\beta_0+\beta_1\xi+\cdots+\beta_{n-1}\xi^{n-1}+\beta_n\xi^n,
\end{align*}
with $\beta_i=\alpha_{n-1}=a_i+b_iT\in{\ma F}_q[T]$, $0\leq i\leq n$
and $\beta_0=a_0$.

Thus
\begin{align}\label{Eq2}
x=\xi^nG(T)&=\sum_{i=0}^n a_i\xi^i+T\sum_{i=1}^n b_i\xi^i\nonumber\\
&=\sum_{i=0}^n a_i\xi^i+(T\xi)\sum_{i=0}^{n-1}b_{i+1}\xi^i=F(\xi)+(T\xi)H(\xi)
\end{align}
with $F(\xi), G(\xi)\in R_{\xi}$, $\deg F(\xi)\leq n$, $\deg H(\xi)\leq n-1$.

Note that the degree of $F(\xi)$ in $T$ is even and the degree of $(T\xi)H(\xi)$
is odd, so that it follows that $x=0\iff F(\xi)=H(\xi)=0$. In particular
$\{1,T\xi\}$ is an integral basis of $A/R_{\xi}$. On the other hand,
since $\xi=\frac{1}{T^2+1}$, it follows that $(\xi T)^2=-\xi^2+\xi$.
Therefore 
\[
\ell(Z):=\Irr(Z,T\xi,{\ma F}_q(\xi))=Z^2+\xi^2-\xi.
\]

Let ${\ma F}_q[X,Y]\xrightarrow[\phantom{xxxx}]{\phi} A$ 
be given by $\phi(f(X,Y))=
f(\xi,T\xi)$. From (\ref{Eq2}), we obtain that $\phi$ is a ring epimorphism.
Further $\phi(Y^2+X^2-X)=0$, that is, $\langle Y^2+X^2-X\rangle
\subseteq \ker \phi$ and $\phi$ induces the epimorphism $\tilde{\phi}\colon
{\ma F}_q[X,Y]/\langle Y^2+X^2-X\rangle\longrightarrow A$ given by
$\tilde{\phi}(f(X,Y)\bmod \langle Y^2+X^2-X\rangle) = f(\xi,T\xi)$.
From (\ref{Eq2}) follows that $\tilde{\phi}$ is a ring isomorphism.

We may apply Kummer's Theorem on the decomposition of prime\
ideals in the ring extension $A/R_{\xi}$ since $A=R_{\xi}[T\xi]$.
In particular we have
\begin{gather*}
\ell(Z)\bmod \xi=Z^2;\qquad \ell(Z)\bmod (\xi-1)=Z^2;
\intertext{so that}
(\xi)={\eu p}_{\xi}^2\quad \text{with}\quad {\eu p}_{\xi}=(\xi,T\xi)\quad\text{and}\\
(\xi-1)={\eu p}_{\xi-1}^2\quad \text{with}\quad {\eu p}_{\xi-1}=
(\xi-1,T\xi),
\intertext{with ${\eu p}_{\xi}$ and ${\eu p}_{\xi-1}$
prime ideals of $A$. Furthermore, $(T\xi)^2=\xi(1-\xi)$, so that}
(T\xi)={\eu p}_{\xi} {\eu p}_{\xi-1}.
\end{gather*}
Thus, $\xi,\xi-1$ and $T\xi$ are irreducible non prime elements of $A$
and $(T\xi)^2=\xi(1-\xi)$.

Let $\rho\colon A\to E\langle \tau\rangle$ be a rank one
$A$--Drinfeld module, where $E=K_{\rho}=H_A$ is the field of 
definition of $\rho$. In fact, since $\deg \xi=\deg T\xi=2$,
$\rho$ is determined by
\begin{gather*}
\rho_{\xi}=\xi + \gamma_1\tau+\gamma_2 \tau^2;\qquad
\rho_{T\xi}=T\xi+\epsilon_1\tau+\epsilon_2 \tau^2
\intertext{and since $(T\xi)^2=\xi(1-\xi)$, we obtain}
\rho_{(T\xi)^2}=\rho_{T\xi} \rho_{T\xi}=\rho_{\xi}(1-\rho_{\xi})=\rho_{
\xi(1-\xi)}.
\end{gather*}

Let $L:=E(\lambda_{\xi})=E(\rho[\xi])$. Then $L/E$ satisfies that
$\Gal(L/E)\cong \big(A/(\xi)\big)^{\ast}=\big(A/({\eu p}_{\xi})^2\big)^{\ast}$
(Corollary \ref{C3.4}) and $\mu_{\rho}(L)=\rho[\xi]$ (Proposition \ref{P3.5}).
Consider the element
\[
\delta:=\rho_{T\xi}(\lambda_{\xi}).
\]
Then $E\subseteq E(\delta)\subseteq L$. Now, $G:=\Gal(L/E)$
may be identified with $\big(A/(\xi)\big)^{\ast}$ as follows. Since
$A\cong {\ma F}_q[X,Y]/\langle Y^2+X^2-X\rangle$, where $X$ is identified
with $\xi$ and $Y$ with $T\xi$, and the group identification is done by its
action on $\lambda_{\xi}$, then we may write
$G=\{\sigma_{U}\}_{U\in \{1,2,1+y,2+y,1+2y,2+2y\}}$
with $y=Y\bmod \langle Y^2+X^2-X\rangle=T\xi$
and $\sigma_U(\lambda_{\xi}):=\rho_U(\lambda_{\xi})$.

We have $G\cong C_6$ the cyclic group of $6$
elements and generated by
$2+y$. Further $(2+y)^2=1+y$, that is,
the subgroup of $G$ of order
$2$ is generated by $1+y$. Note that
\begin{align*}
\sigma_{1+y}(\delta)&=\rho_{1+T\xi}(\rho_{T\xi}(\lambda_{\xi}))=
\rho_{T\xi+(T\xi)^2}(\lambda_{\xi})=\rho_{T\xi}(\lambda_{\xi})+
\rho_{\xi(1-\xi)}(\lambda_{\xi})=\delta
\intertext{and}
\sigma_{2+y}(\delta)&=\rho_{2+T\xi}(\rho_{T\xi}(\lambda_{\xi}))=
\rho_{2 T\xi +(T\xi)^2}(\lambda_{\xi})\\
&=\rho_{2 T\xi}(\lambda_{\xi})+
\rho_{\xi(1-\xi)}(\lambda_{\xi})=2\delta+0\neq \delta,
\end{align*}
so that $E(\delta)$ is the fixed field of the subgroup $C_3$ of $G$ and
therefore $[E(\delta):E]=2$ and $[L:E(\delta)]=3$.

Note that $E(\delta)/E$ is a cyclotomic coradical extension and of
prime degree $2\neq p=q$. On the other hand, the subextension $L/E$ is
not pure and it is of degree $3=p=q$. All this is contrary to what we 
established in Proposition \ref{cogalois_disjuntodrinfeld}, Corollary
\ref{cogalois_disjuntocor1drinfeld}, Lemma
\ref{pureza_pdrinfeld} and to Theorems \ref{tdim_p_galoisdrinfeld},
\ref{lemadrinfeldextsimcogalois} and \ref{tdim_pdrinfeld}. In other words,
in its actual form, we do not have a cogalois theory for $A$--Drinfeld
modules of rank one if $h_A>1$.

\section{The Carlitz module}\label{S9}

In this section we consider the Carlitz module, more precisely
we are interested in computing the cardinality of the module
$\cog(L/K)$ where $L=k(T)(\Lambda_{P^{n}})$,
$K=k(\Lambda_{P})$ and $P\in R_{T}$ 
is a monic irreducible polynomial. We also assume that
$\car {\mathbb{F}}_{q}=p>2$. 
The goal of this section is to understand why it is so hard
to find torsion elements other than the obvious ones ($\Lambda_{P^n}$).

We have established the existence of a group isomorphism
\[
\phi:\cog(L/K)\rightarrow Z^{1}(G,M)
\]
where $G=\Gal(L/K)$ and $M=\Lambda_{P^{n}}$.
Therefore $\phi(\alpha+K)$ is a crossed homomorphism
defined as
\[
\phi(\alpha+K)(\sigma)=\sigma(\alpha)-\alpha
\]
for each $\sigma\in G$. To simplify the 
notation, frequently we will write
$\phi(\alpha)$ instead of $\phi(\alpha+K)$.
On the other hand if we have a crossed homomorphism
$f:G\rightarrow M$ it satisfies
\begin{equation}\label{defhomocruzado}
f(\sigma\cdot\tau)=f(\sigma)+\sigma\cdot f(\tau)
\end{equation}
for each $\sigma,\tau\in G$. We will understand that
$\sigma\cdot f(\tau)$ as the action of $\sigma$
on $f(\tau)$. Since the elements of $M$
are of the form ${\mathrm{C}}_{D}(\lambda_{P^{n}})$ we may write
\[
f(\sigma)={\mathrm{C}}_{D_{\sigma}}(\lambda_{P^{n}}).
\]

Further, by the division algorithm, we may assume that
$D_{\sigma}$ is a polynomial of degree less than or equal
to $\deg P^{n}=n\deg P$. Note that with the exponents of the
elements $\sigma$ it is possible to form a system of equations
using relation (\ref{defhomocruzado}) as follows: since
$f(\sigma\cdot\tau)={\mathrm{C}}_{D_{\sigma\cdot\tau}}
(\lambda_{P^{n}})$
we have 
\begin{equation}\label{defhomocruzado2}
{\mathrm{C}}_{D_{\sigma\cdot\tau}}(\lambda_{P^{n}})=
{\mathrm{C}}_{D_{\sigma}}(\lambda_{P^{n}})+
\sigma\cdot{\mathrm{C}}_{D_{\tau}}(\lambda_{P^{n}}).
\end{equation}
Now $\sigma$ is of the form
$\sigma=1+B_{\sigma}P^{s}$
with $\gcd(B_{\sigma},P)=1$ and $1\leq s\leq n-1$ (see \cite{Lam}). 
Therefore equation (\ref{defhomocruzado2}) can be rewritten as:
\begin{equation}\label{defhomocruzado3}
{\mathrm{C}}_{D_{\sigma\cdot\tau}}(\lambda_{P^{n}})
={\mathrm{C}}_{D_{\sigma}}(\lambda_{P^{n}})+
\sigma\cdot{\mathrm{C}}_{D_{\tau}}(\lambda_{P^{n}}).
\end{equation}
Since the group $G$ is commutative and $\tau=1+B_{\tau}P^{t}$, 
with $\gcd(B_{\tau},P)=1$ and $1\leq t\leq n-1$, we obtain
\begin{equation}\label{defhomocruzado4}
{\mathrm{C}}_{D_{\sigma\cdot\tau}}(\lambda_{P^{n}})
={\mathrm{C}}_{D_{\tau}}(\lambda_{P^{n}})+
{\mathrm{C}}_{(1+B_{\tau}P^{t})D_{\sigma}}\lambda_{P^{n}}.
\end{equation}

Therefore the exponents satisfy the system of 
equations (modulo $P^{n}$):
\begin{eqnarray}\label{defhomocruzado5}
D_{\sigma\cdot\tau}=D_{\sigma}+(1+B_{\sigma}P^{s})D_{\tau},\\
D_{\sigma\cdot\tau}=D_{\tau}+(1+B_{\tau}P^{t})D_{\sigma}.
\end{eqnarray}

If $\tau=\sigma^{-1}$, we have
\begin{equation}\label{defhomocruzado6}
0=D_{\sigma}+(1+B_{\sigma}P^{s})D_{\tau}.
\end{equation}

Equation (\ref{defhomocruzado6}) allow us to obtain the
possible solutions of system (\ref{defhomocruzado5}) and
therefore the number of elements of $\cog (L/K)$. This can
be used as a first rough algorithm to solve the system of equations.
First we obtain the multiplication table of the group $G$,
then we parameterize the possible solutions in a vector with
components
\[
D_{\sigma}=-(1+B_{\sigma}P^{s})D_{\tau}
\]
and finally we verify which of these vectors are really solutions
of system (\ref{defhomocruzado5}). One of the most important
problems we have with this approach is that, even for very small
prime numbers $p$, the number of solutions is very large.

This approach allows us to find explicit torsion elements other
than the class of $\lambda_{P^{n}}$. We achieve this goal
finding the complete list of crossed homomorphisms and
then we use the function $\phi$ to find the values we are 
interested in, that is, the torsion points of the module $\cog(L/K)$.

\begin{ejemplo}\label{E9.1}
We may apply the previous approach to the following case.
Let $q=p=3$, $P=T$ and $n=2$. We compute first the multiplication
table of  $G$:
\begin{center}
\begin{tabular}{l | c  r  r}
 & $\sigma_{1}$ & $\sigma_{2}$ & $\sigma_{3}$ \\
\hline
$\sigma_{1}$ & $\sigma_{1}$ & $\sigma_{2}$ & $\sigma_{3}$\\
$\sigma_{2}$ & $\sigma_{2}$ & $\sigma_{3}$ & $\sigma_{1}$\\
$\sigma_{3}$ & $\sigma_{3}$ & $\sigma_{1}$ & $\sigma_{2}$\\
\end{tabular}
\end{center}
where $\sigma_{1}=1$, $\sigma_{2}=1+T$ and $\sigma_{3}=1+ 2T$
modulo $T^2$, are all the elements of $G$. Therefore, in this case, the
system equations (\ref{defhomocruzado5})
can be written as follows:
\begin{eqnarray}\label{defhomocruzado7}
D_{\sigma_{3}}&=&D_{\sigma_{2}}+(1+T)D_{\sigma_{2}},\\
D_{\sigma_{2}}&=&-(1+T)D_{\sigma_{3}},\\
D_{\sigma_{2}}&=&D_{\sigma_{3}}+(1+2T)D_{\sigma_{3}}.
\end{eqnarray}

Therefore the solutions of the previous system are:
\[
(T^{2},(2T+2)D_{\sigma_{3}},D_{\sigma_{3}}),
\]
where $D_{\sigma_{3}}$ runs through all the polynomials of degree
less than or equal to $1$, with coefficients in ${\mathbb{F}}_{q}$.
Note that the first component of the previous vector is $T^2$ since
we know that if $f$ is a crossed homomorphism, we have $f(1)=0$.
Therefore the number of solutions of the system 
(\ref{defhomocruzado7}) is $9$. This was already obtained in
\cite[Example 7.8]{sanchez}.

Using the function $\phi$ it can be shown that
$\phi(\lambda_{T^{2}})=f_{4}$ and 
$\phi({\mathrm{C}}_{2}(\lambda_{T^{2}}))=f_{7}$. 
We want to find the rest of the $\alpha_{i}$. 
To achieve this, note that
\[
\alpha_{2}=a_{0}+a_{1}\lambda_{T^{2}}+a_{2}({\mathrm{C}}_{2}
(\lambda_{T^{2}}))^{2},
\]
where $a_{i}\in K$ for $i=0,1,2$. Now
\begin{align*}
\phi(\alpha_{2})(\sigma_{1})&= 0,\\
\phi(\alpha_{2})(\sigma_{2})&= a_{1}({\mathrm{C}}_{1+T}(\lambda_{T^{2}})
-\lambda_{T^{2}})+a_{2}(({\mathrm{C}}_{1+T}(\lambda_{T^{2}}))^{2}-
(\lambda_{T^{2}})^{2}),\\
\phi(\alpha_{2})(\sigma_{3})&= a_{1}({\mathrm{C}}_{1+2T}(\lambda_{T^{2}
})-\lambda_{T^{2}})+a_{2}(({\mathrm{C}}_{1+2T}\lambda_{T^{2}})^{2}-
(\lambda_{T^{2}})^{2}).
\end{align*}

On the other hand, we want that $\phi(\alpha_{2})=f_{2}$.
Hence, we obtain the system of equations:
\begin{align*}
a_{1}({\mathrm{C}}_{1+T}(\lambda_{T^{2}})-\lambda_{T^{2}})+a_{2}
(({\mathrm{C}}_{1+T}(\lambda_{T^{2}}))^{2}-(\lambda_{T^{2}})^{2})&=
 {\mathrm{C}}_{2T+2}(\lambda_{T^{2}}),\\
a_{1}({\mathrm{C}}_{1+2T}(\lambda_{T^{2}})-\lambda_{T^{2}})+
a_{2}(({\mathrm{C}}_{1+2T}(\lambda_{T^{2}}))^{2}-(\lambda_{T^{2}})^{2})
&= \lambda_{T^{2}}.
\end{align*}

The solutions of the system are:
\begin{align*}
{a_1}&=\dfrac{\left|\begin{array}{cc}
 {\mathrm{C}}_{2T+2}(\lambda_{T^{2}}) & 
 ({\mathrm{C}}_{1+T}(\lambda_{T^{2}}))^{2}-(\lambda_{T^{2}})^{2} \\
\lambda_{T^{2}} & ({\mathrm{C}}_{1+2T}
(\lambda_{T^{2}}))^{2}-(\lambda_{T^{2}})^{2} \\
\end{array}
\right|}{\left|\begin{array}{cc}
{\mathrm{C}}_{1+T}(\lambda_{T^{2}})-\lambda_{T^{2}} 
& ({\mathrm{C}}_{1+T}(\lambda_{T^{2}}))^{2}-(\lambda_{T^{2}})^{2} \\
{\mathrm{C}}_{1+2T}(\lambda_{T^{2}})-
\lambda_{T^{2}}& ({\mathrm{C}}_{1+2T}
(\lambda_{T^{2}}))^{2}-(\lambda_{T^{2}})^{2} \\
\end{array}
\right|},\\
{a_2}&=\dfrac{\left|\begin{array}{cc}
{\mathrm{C}}_{1+T}(\lambda_{T^{2}})-\lambda_{T^{2}} 
& {\mathrm{C}}_{2T+2}(\lambda_{T^{2}}) \\
{\mathrm{C}}_{1+2T}(\lambda_{T^{2}})-\lambda_{T^{2}}
& \lambda_{T^{2}} \\
 \end{array}
\right|}{\left|\begin{array}{cc}
{\mathrm{C}}_{1+T}(\lambda_{T^{2}})-\lambda_{T^{2}}
& ({\mathrm{C}}_{1+T}(\lambda_{T^{2}}))^{2}-
(\lambda_{T^{2}})^{2} \\
 {\mathrm{C}}_{1+2T}(\lambda_{T^{2}})-\lambda_{T^{2}}& 
 ({\mathrm{C}}_{1+2T}(\lambda_{T^{2}}))^{2}-(\lambda_{T^{2}})^{2} \\
\end{array}
\right|}.
\end{align*}

Proceeding similarly, it is possible to find the rest of the
 $\alpha_{i}$.
\end{ejemplo}

The method of Example \ref{E9.1} can be used to other situations,
more precisely, to Galois extensions $L/K$. In some cases it is
possible to describe the lattice of radical extensions.

\begin{ejemplo}
Let $q=p=3$, $P=T$ and $n=3$.The multiplication table of the
group $G$ is:
\begin{center}
\begin{tabular}{l | c r r r r r r r r}
 & $\sigma_{1}$ & $\sigma_{2}$ & $\sigma_{3}$ 
 & $\sigma_{4}$ & $\sigma_{5}$ & $\sigma_{6}$ 
 & $\sigma_{7}$ & $\sigma_{8}$ & $\sigma_{9}$\\
\hline
$\sigma_{1}$ & $\sigma_{1}$ & $\sigma_{2}$ 
& $\sigma_{3}$ & $\sigma_{4}$ & $\sigma_{5}$ 
& $\sigma_{6}$ & $\sigma_{7}$ & $\sigma_{8}$ 
& $\sigma_{9}$\\
$\sigma_{2}$ & $\sigma_{2}$ & $\sigma_{3}$ 
& $\sigma_{1}$ & $\sigma_{5}$ & $\sigma_{9}$ 
& $\sigma_{7}$ & $\sigma_{8}$ & $\sigma_{6}$ 
& $\sigma_{4}$\\
$\sigma_{3}$ & $\sigma_{3}$ & $\sigma_{1}$ 
& $\sigma_{2}$ & $\sigma_{9}$ & $\sigma_{4}$ 
& $\sigma_{8}$ & $\sigma_{6}$ & $\sigma_{7}$ 
& $\sigma_{5}$\\
$\sigma_{4}$ & $\sigma_{4}$ & $\sigma_{5}$ 
& $\sigma_{9}$ & $\sigma_{7}$ & $\sigma_{8}$ 
& $\sigma_{3}$ & $\sigma_{1}$ & $\sigma_{2}$ 
& $\sigma_{6}$\\
$\sigma_{5}$ & $\sigma_{5}$ & $\sigma_{9}$ 
& $\sigma_{4}$ & $\sigma_{8}$ & $\sigma_{6}$ 
& $\sigma_{1}$ & $\sigma_{2}$ & $\sigma_{3}$ 
& $\sigma_{7}$\\
$\sigma_{6}$ & $\sigma_{6}$ & $\sigma_{7}$ 
& $\sigma_{8}$ & $\sigma_{3}$ & $\sigma_{1}$ 
& $\sigma_{5}$ & $\sigma_{9}$ & $\sigma_{4}$ 
& $\sigma_{2}$\\
$\sigma_{7}$ & $\sigma_{7}$ & $\sigma_{8}$ 
& $\sigma_{6}$ & $\sigma_{1}$ & $\sigma_{2}$ 
& $\sigma_{9}$ & $\sigma_{4}$ & $\sigma_{5}$ 
& $\sigma_{3}$\\
$\sigma_{8}$ & $\sigma_{8}$ & $\sigma_{6}$ 
& $\sigma_{7}$ & $\sigma_{2}$ & $\sigma_{3}$ 
& $\sigma_{4}$ & $\sigma_{5}$ & $\sigma_{9}$ 
& $\sigma_{1}$\\
$\sigma_{9}$ & $\sigma_{9}$ & $\sigma_{4}$ 
& $\sigma_{5}$ & $\sigma_{6}$ & $\sigma_{7}$ 
& $\sigma_{2}$ & $\sigma_{3}$ & $\sigma_{1}$ 
& $\sigma_{8}$\\
\end{tabular}
\end{center}
where $\sigma_{1}=1$, $\sigma_{2}= 2T^{2}+T+1$, 
$\sigma_{3}=2T^{2}+2T+1$, $\sigma_{4}= T^{2}+T+1$, 
$\sigma_{5}= T^{2}+2T+1$, $\sigma_{6}= T+1$, 
$\sigma_{7}=2T+1$ , $\sigma_{8}=T^{2}+1$ 
and $\sigma_{9}=2T^{2}+1$.

Using Matlab all the solutions were obtained and
\[
|\cog(L/K)|=27=3^{3}.
\]

A refinement of \cite[Proposition 7.2]{sanchez} allows us
to find all the subextensions $L^{\prime}/K$ of $L/K$
that are simple radical, namely:
\begin{gather*}
K_{1}=L^{\{\sigma_{1},\sigma_{8},\sigma_{9}\}},\ 
K_{2}=L^{\{\sigma_{1},\sigma_{5},\sigma_{6}\}},\ 
K_{3}=L^{\{\sigma_{1},\sigma_{2},\sigma_{3}\}},\ 
\text{and}\ 
K_{4}=L^{\{\sigma_{1},\sigma_{4},\sigma_{7}\}}.
\end{gather*}
\end{ejemplo}

\end{document}